\documentclass{amsart}
\usepackage{amssymb}
\usepackage{amsmath}
\usepackage{amsfonts}

\newtheorem{theorem}{Theorem}
\theoremstyle{plain}

\newtheorem{definition}{Definition}

\newtheorem{lemma}{Lemma}

\newtheorem{remark}{Remark}

\numberwithin{equation}{section}
\input{tcilatex}

\begin{document}
\title{$(\kappa ,\mu ,\upsilon =const.)$-CONTACT\ METRIC\ MANIFOLDS\ WITH\ $%
\xi (I_{M})=0$}
\author{I. K\"{u}peli Erken}
\address{Art and Science Faculty,Department of Mathematics, Uludag
University, 16059 Bursa, TURKEY}
\email{iremkupeli@uludag.edu.tr}
\author{C. Murathan}
\address{Art and Science Faculty,Department of Mathematics, Uludag
University, 16059 Bursa, TURKEY}
\email{cengiz@uludag.edu.tr}
\date{13 june}

\begin{abstract}
We give a local classification of $(\kappa ,\mu ,\upsilon =const.)$-contact
metric manifold $(M,\phi ,\xi ,\eta ,g)$ with $\kappa <1$ which satisfies
the condition " the Boeckx invariant function $I_{M}=\frac{1-\frac{\mu }{2}}{%
\sqrt{1-\kappa }}$\ \ is constant along the integral curves of the
characteristic vector field $\xi $".
\end{abstract}

\maketitle

\section{\textbf{Introduction}}

It is well known that there exist contact Riemannian manifolds $%
(M^{2n+1},\phi ,\xi ,\eta ,g)$ for which the curvature tensor $R$ in the
direction of characteristic vector field $\xi $ satisfies $R(X,Y)\xi =0,$
for all $X,Y\in \Gamma (TM)$ $.$ For example, the tangent sphere bundle of a
flat Riemannian manifold carries such a structure. In \cite{BL2} Blair
studied for the first time the class of contact metric manifolds satisfying
above condition. If one applies a ${\mathcal{D}}_{\alpha }$-homothetic
deformation on $M^{2n+1}$ with $R(X,Y)\xi =0$, one can find a new class of
contact metric manifolds satisfying 
\begin{equation}
R(X,Y)\xi =\kappa \left( \eta \left( Y\right) X-\eta \left( X\right)
Y\right) +\mu \left( \eta \left( Y\right) hX-\eta \left( X\right) hY\right) ,
\label{k,mu}
\end{equation}%
for some constants $\kappa $ and $\mu $, where $2h$ denotes the Lie
derivative of the structure tensor $\phi $ with respect to characteristic
vector field $\xi $. A contact metric manifold belonging to this class is
called $(\kappa ,\mu )$-contact metric manifold. This new class of
Riemannian manifolds was introduced in \cite{BKP} as a natural
generalization both of $R(X,Y)\xi =0$ and the Sasakian condition $R(X,Y)\xi
=\eta \left( Y\right) X-\eta \left( X\right) Y$. Nowadays contact $(\kappa
,\mu )$-manifolds are considered a very important topic in contact
Riemannian geometry. In fact in despite of the technical appearance of the
definition, there are good reasons for studying $(\kappa ,\mu )$-spaces. The
first is that, in the non-Sasakian case (that is for $\kappa \neq 1$), the
condition (\ref{k,mu}) determines the curvature tensor field completely;
next, $(\kappa ,\mu )$-spaces provide non-trivial examples of some
remarkable classes of contact Riemannian manifolds, like \emph{CR-integrable}
contact metric manifolds (\cite{Tan2}), $H$-contact manifolds (\cite{PE}), 
\emph{harmonic} contact metric manifolds (\cite{vergara1}), or contact
Riemannian manifolds with $\eta $-parallel tensor (\cite{BCH}); moreover, a
local classification is known (\cite{BO}) and while the values of $\kappa $
and $%
\mu
$ change, the form of (\ref{k,mu}) is invariant under ${\mathcal{D}}_{\alpha
}$-homothetic deformations \cite{BKP}. Finally, there are also non-trivial
examples of $(\kappa ,\mu )$-contact metric manifolds, the most important
being the unit tangent sphere bundle of a Riemannian manifold of constant
sectional curvature with the usual contact metric structure.

In \cite{BO} Boeckx provided a local classification of non-Sasakian $(\kappa
,\mu )$-contact metric manifold \ respect to the number 
\begin{equation}
I_{M}=\frac{1-\frac{\mu }{2}}{\sqrt{1-\kappa }},  \label{Boeckx invariant}
\end{equation}%
which is an invariant of a $(\kappa ,\mu )$-contact metric manifold up to ${%
\mathcal{D}}_{\alpha }$-homothetic deformations.

Koufogiorgos and Tsichlias \cite{KO} proved the existence of a new class $3$%
-dimensional contact metric manifolds which are called generalized $(\kappa
,\mu )$-contact metric manifolds. Such a manifold satisfies the (\ref{k,mu})
and $\kappa ,%
\mu
$ \ are non constant smooth functions on $M.$ Moreover, it is showed in \cite%
{KO} that if $n>1$, then $\kappa $ and $\mu $ are necessarily constant.

In \cite{KMP} the condition (\ref{k,mu} ) is generalized as 
\begin{eqnarray}
R(X,Y)\xi &=&\kappa \left( \eta \left( Y\right) X-\eta \left( X\right)
Y\right) +\mu \left( \eta \left( Y\right) hX-\eta \left( X\right) hY\right)
\label{k,mu,vu} \\
&&+\upsilon \left( \eta \left( Y\right) \phi hX-\eta \left( X\right) \phi
hY\right) ,  \notag
\end{eqnarray}%
where $\kappa ,%
\mu
$ and $\upsilon $ \ are non constant smooth functions on $M.$ If the
curvature tensor field of the Levi-Civita connection on $M$ satisfies (\ref%
{k,mu,vu}), we say $(M^{2n+1},\phi ,\xi ,\eta ,g)$ is \ a $(\kappa ,%
\mu
,\upsilon )$-contact metric manifolds. Also, it is proved that, for
dimensions greater than three, such manifolds are reduced to $(\kappa ,\mu )$%
-contact metric manifolds whereas, in three dimensions, $(\kappa ,%
\mu
,\upsilon )$ -contact metric manifolds.

Koufogiorgos and Tsichlias \cite{KO2} gave a local classification of a
non-Sasakian generalized $(\kappa ,%
\mu
)$-contact metric manifold which satisfies the condition " the function $\mu 
$ is constant along the integral \ curves of the characteristic vector field 
$\xi ,$ i.e. $\xi (\mu )=0$". One can easily prove that this condition is
equivalent to $\xi (I_{M})=0$ for a non-Sasakian generalized $(\kappa ,%
\mu
)$-contact metric manifold. This has been our motivation for studying
non-Sasakian $(\kappa ,%
\mu
,\upsilon )$ -contact metric manifolds with $\xi (I_{M})=0$. We can prove
that \ $\xi (I_{M})=0$ satisfies the condition $\xi (\mu )=\upsilon (\mu -2)$%
. Moreover, the converse is also true.

The paper is organized as follows. Section 2 contains some necessary
background on contact metric manifolds. In Section 3, we give some result
concerning $(\kappa ,\mu ,\upsilon )$ -contact metric manifolds. In the last
section, we locally classify $(\kappa ,%
\mu
,\upsilon =const.)$-contact metric manifold with $\xi (I_{M})=0.$ All
manifolds are assumed to be connected.

\section{Preliminaries}

A differentiable manifold $M$ of dimension $2n+1$ is said to be a contact
manifold if it carries a global 1-form $\eta $ such that $\eta \wedge (d\eta
)^{n}\neq 0$. It is well known that then there exists a unique vector field $%
\xi $ (called the Reeb vector field) such that $\eta (\xi )=1$ and $d\eta
(\xi ,\cdot )=0$. Any contact manifold $(M,\eta )$ admits a Riemannian
metric $g$ and a $(1,1)$-tensor field $\phi $ such that 
\begin{gather}
\phi ^{2}=-I+\eta \otimes \xi ,\ \phi \xi =0,\ \eta (X)=g(X,\xi )  \label{B1}
\\
g(\phi X,\phi Y)=g(X,Y)-\eta (X)\eta (Y),\ g(X,\phi Y)=d\eta (X,Y),
\label{B0}
\end{gather}%
for any vector field $X$ and $Y$ on $M$. Define an operator $h$ by $h=\frac{1%
}{2}{\mathcal{L}}_{\xi }\phi $, where $\mathcal{L}$ denotes Lie
differentiation. The tensor field $h$ vanishes identically if and only if
the vector field $\xi $ is Killing and in this case the contact metric
manifold is said to be K-contact. It is well known that $h$ and $\phi h$ are
symmetric operators, $h$ anti-commutes with $\phi $ 
\begin{equation}
\phi h+h\phi =0,\text{ }h\xi =0,\text{ }\eta \circ h=0,\text{ }trh=\text{ }%
tr\phi h=0,  \label{B1a}
\end{equation}%
where $trh$ denotes the trace of $h$. Since $h$ anti-commutes with $\phi $,
if $X$ is an eigenvector of $h$ corresponding to the eigenvalue $\lambda $
then $\phi X$ is also an eigenvector of $h$ corresponding to the eigenvalue $%
-\lambda $ \cite{Tan}. Moreover, for any contact manifold $M$, the following
is satisfied%
\begin{equation}
\nabla _{X}\xi =-\phi X-\phi hX  \label{B2}
\end{equation}%
where $\nabla $ is the Riemannian connection of $g$. If a contact metric
manifold $M$ is normal (i.e., $N_{\phi }+2d\eta \otimes \xi =0,$ where $%
N_{\phi }$ denotes the Nijenhuis tensor formed with $\phi $), then $M$ is
called a Sasakian manifold. Equivalently, a contact metric manifold is
Sasakian if and only if $R(X,Y)\xi =\eta (Y)X-\eta (X)Y$. Moreover, any
Sasakian manifold is $K$-contact and in 3-dimension the converse also holds%
\textit{\ } \cite{B1}.

As a generalization of both $R(X,Y)\xi =0$ and the Sasakian case consider

\begin{equation}
R(X,Y)\xi =\kappa (\eta (Y)X-\eta (X)Y)+\mu (\eta (Y)hX-\eta (X)hY)
\label{k,mu,0}
\end{equation}%
for constants $\kappa $ and $\mu $. This condition is called $(\kappa ,\mu )$%
- nullity condition. This kind of manifold is called $(\kappa ,\mu )$%
-contact metric manifold\textit{\ }which was introduced and deeply studied
by Blair, Koufogiorgos and Papantoniou in \cite{BKP}.

The standard contact metric structure on the tangent sphere bundle $T_{1}M$
satisfies the $(\kappa,\mu)$-nullity condition if and only if the base
manifold $M$ is of constant curvature. In particular if $M$ has constant
curvature $c$, then $\kappa=c(2-c)$ and $\mu=-2c$.

Given a non-Sasakian $(\kappa ,\mu )$-contact metric manifold\textit{\ } $M$%
, Boeckx \cite{BO} introduced an invariant $I_{M}:=\frac{1-\frac{\mu }{2}}{%
\sqrt{1-\kappa }}$, and proved that two non-Sasakian $(\kappa ,\mu )$-
contact metric manifolds $(M_{1},\phi _{1},\xi _{1},\eta _{1},g_{1})$ and $%
(M_{2},\phi _{2},\xi _{2},\eta _{2},g_{2})$ are locally isometric as contact
metric manifolds if and only if $I_{M_{1}}=I_{M_{2}}$. Then the invariant $%
I_{M}$ was used by Boeckx for providing a full classification of $(\kappa
,\mu )$- contact metric manifolds.

By a generalized $(\kappa ,\mu )$-contact metric manifold\textit{\ }we mean
a $3$-dimensional contact metric manifold such that it satisfies (\ref%
{k,mu,0}), where $\kappa ,\mu $ are smooth non-constant functions on $M.$ A
manifold of this class was studied by Koufogiorgos and Tsichlias in \cite{KO}%
, \cite{KO1} and \cite{KO2}. A recent generalization of the $(\kappa ,\mu )$%
-contact metric manifold is given following definition.

\begin{definition}[\protect\cite{KMP}]
A $(\kappa ,\mu ,\upsilon )$-contact metric manifold is a contact metric
manifold $(M^{2n+1},\phi ,\xi ,\eta ,g)$ on which the Riemannian curvature
tensor satisfies for every $X,Y\in \Gamma (TM)$ the condition 
\begin{eqnarray}
R(X,Y)\xi &=&\kappa \left( \eta \left( Y\right) X-\eta \left( X\right)
Y\right) +\mu \left( \eta \left( Y\right) hX-\eta \left( X\right) hY\right)
\label{k,mu,2} \\
&&+\upsilon \left( \eta \left( Y\right) \phi hX-\eta \left( X\right) \phi
hY\right) ,  \notag
\end{eqnarray}%
where $\kappa ,\mu ,\upsilon $ are smooth functions on $M$.
\end{definition}

A contact metric manifold whose characteristic vector field $\xi $ is a
harmonic vector field is called an $H$-contact manifold. Moreover, in \cite%
{PE} \ Perrone proved \ that $\xi $ is a harmonic vector field if and only
if $\xi $ is an eigenvector of the Ricci operator. In \cite{KMP}
Koufogiorgos, Markellos and Papantoniou characterized the 3-dimensional $H$%
-contact metric manifolds in $(\kappa ,\mu ,\upsilon )$-contact metric
manifolds. In particular, they proved following Theorem.

\begin{theorem}[\protect\cite{KMP}]
Let $(M^{2n+1},\phi ,\xi ,\eta ,g)$ be a $3$-dimensional contact metric
manifold. If $M$ is a $(\kappa ,\mu ,\upsilon )$-contact metric manifold,
then $M$ is an $H$-contact metric manifold. Conversely, if $M$ is a $3$%
-dimensional $H$-contact metric manifold, then $M$ is a $(\kappa ,\mu
,\upsilon )$-contact metric manifold on an everywhere open and dense subset
of $M$.
\end{theorem}

It is proved that for a $(\kappa ,\mu ,\upsilon )$-contact metric manifold $%
M $ of dimension greater than $3$, the functions $\kappa ,\mu $ are
constants and $\upsilon $ is the zero function \cite{KMP}.

Given a contact metric structure $(M^{2n+1},\phi ,\xi ,\eta ,g)$, consider
the deformed structure 
\begin{equation}
\bar{\eta}=\alpha \eta ,\text{ \ \ }\bar{\xi}=\frac{1}{\alpha }\xi ,\text{ \
\ }\bar{\phi}=\phi ,\text{ \ \ }\bar{g}=\alpha g+\alpha (\alpha -1)\eta
\otimes \eta ,  \label{DHOM}
\end{equation}%
where $\alpha $ is a positive constant. This deformation is called ${%
\mathcal{D}}_{\alpha }$-homothetic deformation \cite{Tan2}. It is well known
that $(M^{2n+1},\bar{\phi},\bar{\xi},\bar{\eta},\bar{g})$ is also a contact
metric manifold. By the direct computations we easily see that the tensor $h$
and the curvature tensor transform in the following manner \cite{BKP}; 
\begin{equation}
\bar{h}=\frac{1}{\alpha }h  \label{h}
\end{equation}%
and%
\begin{eqnarray}
\alpha \bar{R}(X,Y)\bar{\xi} &=&R(X,Y)\xi +(\alpha -1)^{2}(\eta (Y)X-\eta
(X)Y)  \notag \\
&&-(\alpha -1)((\nabla _{X}\phi )Y-(\nabla _{Y}\phi )X+\eta (X)(Y+hY)-\eta
(Y)(X+hX)),  \label{RDEHOM}
\end{eqnarray}%
for any $X,Y\in \Gamma (TM)$. Moreover, it is well known (\cite{BKP} or \cite%
{Tan2}) that every $3$-dimensional contact metric manifold satisfies 
\begin{equation}
(\nabla _{X}\phi Y)=g(X+hX,Y)\xi -\eta \left( Y\right) (X+hX).  \label{CR}
\end{equation}%
Using (\ref{RDEHOM}) and (\ref{CR}), we obtain that%
\begin{eqnarray}
\bar{R}(X,Y)\bar{\xi} &=&\frac{\kappa +\alpha ^{2}-1}{\alpha ^{2}}(\bar{\eta}%
(Y)X-\bar{\eta}(X)Y)+\frac{\mu +2(\alpha -1)}{\alpha }(\bar{\eta}(Y)\bar{h}X-%
\bar{\eta}(X)\bar{h}Y)  \notag \\
&&+\frac{\upsilon }{\alpha }(\bar{\eta}(Y)\bar{\phi}\bar{h}X-\bar{\eta}(X)%
\bar{\phi}\bar{h}Y)  \label{RKDEM}
\end{eqnarray}%
for any $X,Y\in \Gamma (TM)$.Thus $(M^{2n+1},\bar{\phi},\bar{\xi},\bar{\eta},%
\bar{g})$ is a $(\bar{\kappa},\bar{\mu},\bar{\upsilon})$-contact metric
manifold with 
\begin{equation}
\bar{\kappa}=\frac{\kappa +\alpha ^{2}-1}{\alpha ^{2}},\text{ \ \ }\bar{\mu}=%
\frac{\mu +2(\alpha -1)}{\alpha },\text{ \ \ }\bar{\upsilon}=\frac{\upsilon 
}{\alpha }\text{.}  \label{(k,mu,vu)DHOM}
\end{equation}

\section{$(\protect\kappa ,\protect\mu ,\protect\upsilon )$-contact metric
manifolds}

In this section, we will give some basic results of $(\kappa ,\mu ,\upsilon
) $-contact metric manifolds.

\begin{lemma}[ \protect\cite{KMP}]
The following relations are satisfied on any $(\kappa ,\mu ,\upsilon )$%
-contact metric manifold $(M^{3},\phi ,\xi ,\eta ,g)$.

\begin{equation}
h^{2}=(\kappa -1)\phi ^{2}\text{, \ \ }\kappa =\frac{Trl}{2}\leq 1,
\label{hsq}
\end{equation}%
\begin{equation}
\xi (\kappa )=2\upsilon (\kappa -1)\text{, }  \label{zetak}
\end{equation}%
\begin{equation}
Q\xi =2\kappa \xi ,  \label{Ricci}
\end{equation}%
\begin{equation}
Q=(\frac{\tau }{2}-\kappa )I+(-\frac{\tau }{2}+3\kappa )\eta \otimes \xi
+\mu h+\upsilon \phi h,\text{ \ \ }\kappa <1  \label{Ricci Q}
\end{equation}%
where $Q$ is the Ricci operator of $M$, $\tau $ denotes scalar curvature of $%
M$ and $l=R(.,\xi )\xi $.
\end{lemma}

\begin{lemma}
Let $(M,\phi ,\xi ,\eta ,g)$ be a $(\kappa ,\mu ,\upsilon )$-contact metric
manifold. Then, for any point $P\in M$, with $\kappa (P)<1$ there exist a
neighbourhood $U$ of $P$ and an $h$-frame on $U$, i.e. orthonormal vector
fields $\xi ,$ $X$, $\phi X$, defined on $U$, such that 
\begin{equation}
hX=\lambda X,\text{ \ \ }h\phi X=-\lambda \phi X\text{, \ \ }h\xi =0\text{,
\ \ }\lambda =\sqrt{1-\kappa }  \label{hfrme}
\end{equation}%
at any point $q\in U$. Moreover, setting $A=X\lambda ,B=\phi X\lambda $ and $%
C=X\upsilon $, $D=\phi X\upsilon $ on $U$ the following formulas are true :%
\begin{equation}
\nabla _{X}\xi =-(\lambda +1)\phi X,\text{ \ \ }\nabla _{\phi X}\xi
=(1-\lambda )X,  \label{eq1}
\end{equation}%
\begin{equation}
\nabla _{\xi }X=-\frac{\mu }{2}\phi X,\text{ \ \ }\nabla _{\xi }\phi X=\frac{%
\mu }{2}X,  \label{eq2}
\end{equation}%
\begin{equation}
\nabla _{X}X=\frac{B}{2\lambda }\phi X,\text{ \ \ }\nabla _{\phi X}\phi X=%
\frac{A}{2\lambda }X,  \label{eq3}
\end{equation}%
\begin{equation}
\nabla _{\phi X}X=-\frac{A}{2\lambda }\phi X+(\lambda -1)\xi ,\text{ \ \ }%
\nabla _{X}\phi X=-\frac{B}{2\lambda }X+(\lambda +1)\xi ,  \label{eq4}
\end{equation}%
\begin{equation}
\lbrack \xi ,X]=(1+\lambda -\frac{\mu }{2})\phi X,\text{ \ \ }[\xi ,\phi
X]=(\lambda -1+\frac{\mu }{2})X,  \label{eq5}
\end{equation}%
\begin{equation}
\lbrack X,\phi X]=-\frac{B}{2\lambda }X+\frac{A}{2\lambda }\phi X+2\xi ,
\label{eq6}
\end{equation}%
\begin{equation}
h\text{ }grad\mu +\phi h\text{ }grad\upsilon =grad\kappa -\xi (\kappa )\xi ,
\label{eq7}
\end{equation}%
\begin{equation}
X\mu =-2A-D,  \label{eq8}
\end{equation}%
\begin{equation}
\phi X\mu =2B+C,  \label{eq9}
\end{equation}%
\begin{equation}
\xi (A)=(1+\lambda -\frac{\mu }{2})B+\upsilon A+\lambda C,  \label{eq10}
\end{equation}%
\begin{equation}
\xi (B)=(\lambda -1+\frac{\mu }{2})A+\upsilon B+\lambda D,  \label{eq11}
\end{equation}
\end{lemma}

\begin{proof}
The proofs of (\ref{eq1})$-$(\ref{eq6}) are given in \cite{KO} and \cite{KO1}%
. In order to prove( \ref{eq8}), we will use well known formula%
\begin{equation*}
\frac{1}{2}grad\text{ }\tau =\sum\limits_{i=1}^{3}(\nabla _{X_{i}}Q)X_{i},
\end{equation*}%
where $\{X_{1}=\xi ,$ $X_{2}=X$, $X_{3}=\phi X\}$. Using (\ref{Ricci Q}) and
(\ref{B2}), since $trh=trh\phi =0$, we have%
\begin{eqnarray}
\sum\limits_{i=1}^{3}(\nabla _{X_{i}}Q)X_{i} &=&\sum\limits_{i=1}^{3}X_{i}(%
\frac{\tau }{2}-\kappa )+\sum\limits_{i=1}^{3}(X_{i}(\mu
)hX_{i}+X_{i}(\upsilon )\phi hX_{i})  \notag \\
&&+\mu \sum\limits_{i=1}^{3}(\nabla _{X_{i}}h)X_{i}+\upsilon
\sum\limits_{i=1}^{3}(\nabla _{X_{i}}\phi h)X_{i}+\xi (-\frac{\tau }{2}%
+3\kappa )\xi  \notag \\
&=&\frac{1}{2}grad\tau -grad\kappa +h\text{ }grad\mu +\phi h\text{ }%
grad\upsilon +\xi (-\frac{\tau }{2}+3\kappa )\xi  \label{eq14} \\
&&+\mu \sum\limits_{i=1}^{3}(\nabla _{X_{i}}h)X_{i}+\upsilon
\sum\limits_{i=1}^{3}(\nabla _{X_{i}}\phi h)X_{i}  \notag
\end{eqnarray}%
From the relations (\ref{hfrme}), (\ref{eq3}) and (\ref{eq4}), we obtain $%
\sum\limits_{i=1}^{3}(\nabla _{X_{i}}h)X_{i}=0$ and $\sum\limits_{i=1}^{3}(%
\nabla _{X_{i}}\phi h)X_{i}=2\lambda ^{2}\xi .$ Using the last relations in (%
\ref{eq14}), one has%
\begin{equation}
\frac{1}{2}grad\text{ }\tau =\frac{1}{2}grad\tau -grad\kappa +h\text{ }%
grad\mu +\phi h\text{ }grad\upsilon +\xi (-\frac{\tau }{2}+3\kappa )\xi
+2\lambda ^{2}\upsilon \xi  \label{eq15}
\end{equation}%
that is%
\begin{equation}
\xi (\kappa )\xi -grad\kappa +h\text{ }grad\mu +\phi h\text{ }grad\upsilon
+\xi (-\frac{\tau }{2}+2\kappa )\xi +2\lambda ^{2}\upsilon \xi =0.
\label{eq16}
\end{equation}%
Since the vector field $\xi (\kappa )\xi -grad\kappa +h$ $grad\mu +\phi h$ $%
grad\upsilon $ is orthogonal to $\xi .$ So, we get (\ref{eq7}). The
equations (\ref{eq8}) and (\ref{eq9}) are immediate consequences of (\ref%
{eq7}).

By virtue of (\ref{zetak}) and (\ref{eq5}), we have%
\begin{eqnarray*}
\xi (A) &=&\xi X\lambda =[\xi ,X]\lambda +X\xi \lambda =(1+\lambda -\frac{%
\mu }{2})\phi X\lambda +\lambda X\upsilon +\upsilon X\lambda \\
\text{ \ \ \ \ \ \ \ \ \ \ \ \ \ \ \ \ \ \ \ \ \ \ \ \ \ \ \ \ \ \ \ \ \ \ \ 
} &=&(1+\lambda -\frac{\mu }{2})B+C\lambda +\upsilon A.
\end{eqnarray*}%
Similarly, the equation (\ref{eq11}) is proved .
\end{proof}

\section{$(\protect\kappa ,\protect\mu ,\protect\upsilon =const.)$-contact
metric manifolds with $\protect\xi (I_{M})=0$}

Koufogiorgos and Tsichlias \cite{KO2} gave a local classification of a
non-Sasakian generalized $(\kappa ,%
\mu
)$-contact metric manifold which satisfies the condition $\xi (\mu )=0$. We
recall the (\ref{Boeckx invariant}). We can easily prove that $\xi (\mu )=0$
if and only if $\xi (I_{M})=0$. Now, we assume that $(M,\phi ,\xi ,\eta ,g)$
be a non-Sasakian $(\kappa ,\mu ,\upsilon )$-contact metric manifold. Using (%
\ref{zetak}) ,we can easily obtain that $\xi (I_{M})=0$ if and only if $\xi
(\mu )=\upsilon (\mu -2)$. This case is also our motivation. If $\upsilon =0$%
, we have classification which is given in \cite{KO2}. Because of this fact
we assume that $\upsilon $ $\neq 0$. Let us concentrate that the value $%
\upsilon $ is constant. Under this assumption, we will give a local
classification of $(\kappa ,\mu ,\upsilon =const)$- contact metric manifold
with $\kappa <1$ satisfying the condition $\xi (I_{M})=0$ in the following
Theorem.

\begin{theorem}[Main Theorem]
Let $(M,\phi ,\xi ,\eta ,g)$ be a non-Sasakian $(\kappa ,\mu ,\upsilon
=const.)$-contact metric manifold and $\xi (I_{M})=0$, where $\upsilon
=const.$ $\neq 0$. Then

$1)$ At any point of $M$, precisely one of the following relations is valid: 
$\mu =2(1+\sqrt{1-\kappa }),$ or $\mu =2(1-\sqrt{1-\kappa })$

$2)$ At any point $P\in M$ there exists a chart $(U,(x,y,z))$ with $P\in
U\subseteq M,$ such that

\ \ \ \ \ \ $i)$ the functions $\kappa ,\mu $ depend only on the variables $%
x $, $z.$

\ \ \ \ \ \ $ii)$ if $\mu =2(1+\sqrt{1-\kappa }),$ $($resp. $\mu =2(1-\sqrt{%
1-\kappa })),$ the tensor fields $\eta $, $\xi $, $\phi $, $g$, $h$ are
given by the relations,%
\begin{equation*}
\xi =\frac{\partial }{\partial x},\text{ \ \ }\eta =dx-adz
\end{equation*}%
\begin{equation*}
g=\left( 
\begin{array}{ccc}
1 & 0 & -a \\ 
0 & 1 & -b \\ 
-a & -b & 1+a^{2}+b^{2}%
\end{array}%
\right) \text{ \ \ \ \ }\left( \text{resp. \ \ }g=\left( 
\begin{array}{ccc}
1 & 0 & -a \\ 
0 & 1 & -b \\ 
-a & -b & 1+a^{2}+b^{2}%
\end{array}%
\right) \right) ,
\end{equation*}%
\begin{equation*}
\phi =\left( 
\begin{array}{ccc}
0 & a & -ab \\ 
0 & b & -1-b^{2} \\ 
0 & 1 & -b%
\end{array}%
\right) \text{ \ \ \ \ }\left( \text{resp. \ \ }\phi =\left( 
\begin{array}{ccc}
0 & -a & ab \\ 
0 & -b & 1+b^{2} \\ 
0 & -1 & b%
\end{array}%
\right) \right) ,
\end{equation*}%
\begin{equation*}
h=\left( 
\begin{array}{ccc}
0 & 0 & -a\lambda \\ 
0 & \lambda & -2\lambda b \\ 
0 & 0 & -\lambda%
\end{array}%
\right) \text{ \ \ \ \ \ }\left( \text{resp. \ \ }h=\left( 
\begin{array}{ccc}
0 & 0 & a\lambda \\ 
0 & -\lambda & 2\lambda b \\ 
0 & 0 & \lambda%
\end{array}%
\right) \right)
\end{equation*}%
with respect to the basis $\left( \frac{\partial }{\partial x},\frac{%
\partial }{\partial y},\frac{\partial }{\partial z}\right) ,$ where $%
a=2y+f(z)$ \ \ (resp.\ $a=-2y+f(z)$), $b=-\frac{y^{2}}{2}\upsilon -y\frac{%
f(z)}{2}\upsilon -\frac{y}{2}\frac{r^{^{\prime }}(z)}{r(z)}+\frac{2}{%
\upsilon }r(z)e^{\upsilon x}+s(z)$ (resp. $b=\frac{y^{2}}{2}\upsilon -y\frac{%
f(z)}{2}\upsilon -\frac{y}{2}\frac{r^{^{\prime }}(z)}{r(z)}+\frac{2}{%
\upsilon }r(z)e^{\upsilon x}+s(z)$)\ $\lambda =\lambda (x,z)=r(z)e^{\upsilon
x}$ \ and $f(z)$, $r(z)$, $s(z)$ are arbitrary smooth functions of $z.$
\end{theorem}

Before the proof of the main Theorem, we will give a Lemma which contains
some necessary relations to prove the main Theorem

\begin{lemma}
Let $(M,\phi ,\xi ,\eta ,g)$ be a non-Sasakian $(\kappa ,\mu ,\upsilon
=const.)$-contact metric manifold .The following formulas are valid.%
\begin{equation}
\xi (A)=(1+\lambda -\frac{\mu }{2})B+\upsilon A,  \label{zetaA}
\end{equation}%
\begin{equation}
\xi (B)=(\lambda -1+\frac{\mu }{2})A+\upsilon B,  \label{ZetaB}
\end{equation}%
\begin{equation}
X\mu =-2A,  \label{Xmu}
\end{equation}%
\begin{equation}
\phi X\mu =2B,  \label{FiXmu}
\end{equation}%
\begin{equation}
\left[ \xi ,\phi grad\lambda \right] =\upsilon (A\phi X-BX)\text{.}
\label{figrad}
\end{equation}

\begin{proof}
Using (\ref{eq10}), (\ref{eq11}) and constant of $\upsilon $, we have the
relations (\ref{zetaA})(\ref{ZetaB}). From (\ref{eq8}) and (\ref{eq9}, \ we
obtain (\ref{Xmu}) and (\ref{FiXmu}). By (\ref{zetak}) and (\ref{B1}), we
have 
\begin{equation}
grad\lambda =AX+B\phi X+\upsilon \lambda \xi ,\text{ \ }\phi grad\lambda
=A\phi X-BX\text{\ .}  \label{grad1}
\end{equation}

Using (\ref{grad1}), (\ref{eq5}), (\ref{zetaA}) and (\ref{ZetaB}), we find

\begin{eqnarray*}
\left[ \xi ,\phi grad\lambda \right] &=&\left[ \xi ,A\phi X-BX\right] \\
&=&(\xi A)\phi X+A\left[ \xi ,\phi X\right] -(\xi B)X-B\left[ \xi ,X\right]
=\upsilon (A\phi X-BX).
\end{eqnarray*}
\end{proof}
\end{lemma}

\begin{proof}[Proof of the Main Theorem:]
Let $\left\{ \xi ,X,\phi X\right\} $ be an $h$-frame, such that%
\begin{equation*}
hX=\lambda X,\text{ \ \ }h\phi X=-\lambda \phi X,\text{ \ \ \ }\lambda =%
\sqrt{1-\kappa }
\end{equation*}

in an appropriate neighborhood of an arbitrary point of $M$. Using the
hypothesis $\xi (I_{M})=0$ (i.e. $\xi (\mu )=\upsilon (\mu -2)$) and (\ref%
{Xmu}), (\ref{FiXmu}), we have the following relations,

\begin{equation}
(\phi grad\lambda )\mu =4AB,  \label{eqlamda1}
\end{equation}%
\begin{equation}
\left[ \xi ,\phi grad\lambda \right] \mu =4\upsilon AB,  \label{eqlamda2}
\end{equation}%
\begin{equation}
\xi (AB)=2\upsilon AB,  \label{zeta(AB)}
\end{equation}%
\begin{equation}
A\xi B+B\xi A=2AB\upsilon ,  \label{zeta AB1}
\end{equation}%
\begin{equation}
A^{2}(\lambda -1+\frac{\mu }{2})+B^{2}(1+\lambda -\frac{\mu }{2})=0\text{.}
\label{zetaAB2}
\end{equation}

Differentiating the relation (\ref{zetaAB2}) with respect to $\xi $ and
using the relations (\ref{zetak}), \ $\xi (\mu )=\upsilon (\mu -2)$, (\ref%
{zetaA}) and (\ref{ZetaB}) we can successively obtain 
\begin{equation}
(1+\lambda -\frac{\mu }{2})(\lambda -1+\frac{\mu }{2})AB=0.  \label{eqlamda3}
\end{equation}%
We distinguish following cases:

\qquad Case I) $M_{1}=\{P\in M\mid A(P)=0$, $B(P)=0$ $\}$, or

\qquad Case II) $M_{2}=\{P\in M\mid A(P)=0$, $B(P)\neq 0\}$, or

\qquad Case III) $M_{3}=\{P\in M\mid A(P)\neq 0$, $B(P)=0$ $\}$, or

\qquad Case IV) $M_{4}=\{P\in M\mid (1+\lambda -\frac{\mu }{2})(P)=0$, $%
(\lambda -1+\frac{\mu }{2})(P)=0\}$, or

\qquad Case V) $M_{5}=\{P\in M\mid (1+\lambda -\frac{\mu }{2})(P)=0$, $%
(\lambda -1+\frac{\mu }{2})(P)\neq 0\}$, or

\qquad Case VI) $M_{6}=\{P\in M\mid (1+\lambda -\frac{\mu }{2})(P)\neq 0$, $%
(\lambda -1+\frac{\mu }{2})(P)=0\}$.

Firstly we will examine the Case I and the Case IV. We assume that the Case
I is true. In this case, by (\ref{eq6}) and (\ref{zetak}), we get $\xi
(\lambda )=\upsilon \lambda =0.$ Since $\upsilon \neq 0$, we obtain that $%
\lambda (P)=0$. This requires$\ $that $\kappa (P)=1$ which is contradiction
with $\kappa (P)<1$. Let us suppose that the Case IV is valid. But in this
situation, we have $\lambda (P)=0$, or equivalently $\kappa (P)=1$, which is
impossible by the assumption of the mainTheorem. Secondly we consider the
Case II. From the formula (\ref{zetaAB2}), we find $(1+\lambda -\frac{\mu }{2%
})(P)=0$ which appeares in the Case IV or the Case V. Similarly, the Case
III is included in the Case IV or the Case VI. Finally, as the Case IV is
impossible we only consider the Case V and the Case VI. Since $M_{5}$ and $%
M_{6}$ disjoint open sets and the Case IV is impossible, we have $M_{5}$ $%
\cup $ $M_{6}=M.$ Due to the fact that $M$ is connected, we conclude that $%
\{M=M_{5}$ and $M_{6}=\varnothing \}$ or $\{M_{5}=\varnothing $ and $%
M_{6}=M\}$. Regarding the Case V we have $\mu =2(1+\lambda )$, or
equivalently $\mu =2(1+\sqrt{1-\kappa })$ at any point $M$. Similarly,
regarding the Case VI we obtain $\mu =2(1-\lambda )=2(1-\sqrt{1-\kappa )}$.
Therefore, $(1)$ is proved. Now, we will examine the cases $\mu =2(1+\sqrt{%
1-\kappa })$ and $\mu =2(1-\sqrt{1-\kappa })$.

Case V: $\mu =2(1+\sqrt{1-\kappa }).$

Let $P\in M$ and $\{\xi ,X,\phi X\}$ be an $h$-frame on a neighborhood $U$
of $P.$ Using the assumption $\mu =2(1+\sqrt{1-\kappa })$ and (\ref{zetaAB2}%
) we obtain $A=0$ and thus the relations (\ref{eq5}) and (\ref{eq6}) reduce
to%
\begin{equation}
\lbrack \xi ,X]=0,  \label{eq4.1}
\end{equation}%
\begin{equation}
\text{\ }[\xi ,\phi X]=2\lambda X,\text{ \ }  \label{eq4.2}
\end{equation}%
\begin{equation}
\text{\ }[X,\phi X]=-\frac{B}{2\lambda }X+2\xi .  \label{eq4.3}
\end{equation}%
Since $[\xi ,X]=0$, the distribution which is spanned by\ $\xi $ and $X$ is
integrable and so for any $q\in V$ there exist a chart $(V,(x,y,z)\}$ at $%
P\in V\subset U$, such that 
\begin{equation}
\xi =\frac{\partial }{\partial x},\text{ \ \ }X=\frac{\partial }{\partial y},%
\text{ \ \ }\phi X=a\frac{\partial }{\partial x}+b\frac{\partial }{\partial y%
}+c\frac{\partial }{\partial z},  \label{eq4.4}
\end{equation}%
where $a$, $b$ and $c$ are smooth functions on $V$. Since $\xi $, $X$ and $%
\phi X$ are lineraly independent we have $c\neq 0$ at any point of $V$. By
using (\ref{eq4.4}), (\ref{zetak}) and $A=0$ we obtain%
\begin{equation}
\frac{\partial \lambda }{\partial x}=\upsilon \lambda \text{ \ \ and \ }%
\frac{\partial \lambda }{\partial y}=0\text{\ .}  \label{eq4.5}
\end{equation}%
From (\ref{eq4.5}) we find 
\begin{equation}
\lambda =r(z)e^{\upsilon x},  \label{eqlamda}
\end{equation}%
where $r(z)$ is smooth function of $z$ defined on $V$. By using (\ref{eq4.2}%
), (\ref{eq4.3}) and (\ref{eq4.4}) we have following partial differential
equations:%
\begin{equation}
\frac{\partial a}{\partial x}=0,\text{ \ }\frac{\partial b}{\partial x}%
=2\lambda ,\text{ \ \ }\frac{\partial c}{\partial x}=0,  \label{eq4.6}
\end{equation}%
\begin{equation}
\frac{\partial a}{\partial y}=2,\text{ \ }\frac{\partial b}{\partial y}=-%
\frac{B}{2\lambda },\text{ \ \ }\frac{\partial c}{\partial y}=0.
\label{eq4.7}
\end{equation}%
From $\frac{\partial c}{\partial x}=\frac{\partial c}{\partial y}=0$ it
follows that $c=c(z)$ and because of the fact that $c\neq 0$, we can assume
that $c=1$ through a reparametrization of the variable $z$. For the sake of
simplicity we will continue to use the same coordinates $(x,y,z),$ taking
into account that $c=1$ in the relations that we have occured. From $\frac{%
\partial a}{\partial x}=0,$ $\frac{\partial a}{\partial y}=2$ we obtain 
\begin{equation*}
a=a(x,y,z)=2y+f(z),
\end{equation*}%
where $f(z)$ is smooth function of $z$ defined on $V$. Differentiating $%
\lambda $ with respect to $\phi X$ and using (\ref{eq4.5}) we have%
\begin{equation}
B=[(2y+f(z))\upsilon r(z)+r^{\prime }(z)]e^{\upsilon x},  \label{eq4.8}
\end{equation}%
where $r^{\prime }(z)=\frac{\partial r}{\partial z}$. By using the relations 
$\frac{\partial b}{\partial x}=2\lambda $, $\frac{\partial b}{\partial y}=-%
\frac{B}{2\lambda }$ and (\ref{eqlamda}) we get 
\begin{equation*}
b=-\frac{y}{2}(y\upsilon +\upsilon f(z)+\frac{r^{\prime }(z)}{r(z)})+\frac{2%
}{\upsilon }r(z)e^{\upsilon x}+s(z),
\end{equation*}%
where $s(z)$ is smooth function of $z$ defined on $V$. We will calculate the
tensor fields $\eta $, $\phi $, $g$ and $h$ with respect to the basis $\frac{%
\partial }{\partial x}$, $\frac{\partial }{\partial y}$, $\frac{\partial }{%
\partial z}$. For the components $g_{ij}$ of the Riemannian metric, using (%
\ref{eq4.4}) we have 
\begin{equation*}
g_{11}=g(\frac{\partial }{\partial x},\frac{\partial }{\partial x})=g(\xi
,\xi )=1,\text{ \ }g_{22}=g(\frac{\partial }{\partial y},\frac{\partial }{%
\partial y})=g(X,X)=1,\text{ }
\end{equation*}%
\begin{equation*}
\text{\ \ }g_{12}=g_{21}=g(\frac{\partial }{\partial x},\frac{\partial }{%
\partial y})=0,
\end{equation*}%
\begin{eqnarray*}
g_{13} &=&g_{31}=g(\frac{\partial }{\partial x},\phi X-a\frac{\partial }{%
\partial x}-b\frac{\partial }{\partial y}) \\
&=&g(\xi ,\phi X)-ag_{11}=-a,
\end{eqnarray*}%
\begin{eqnarray*}
g_{23} &=&g_{32}=g(\frac{\partial }{\partial y},\phi X-a\frac{\partial }{%
\partial x}-b\frac{\partial }{\partial y}) \\
&=&g(X,\phi X)-ag_{12}-bg_{22}=-b,
\end{eqnarray*}%
\begin{eqnarray*}
1 &=&g(\phi X,\phi X)=g_{33}=a^{2}+b^{2}+g_{33}+2abg_{12}+2ag_{13}+2bg_{23}
\\
&=&a^{2}+b^{2}+g_{33}-2a^{2}-2b^{2}=g_{33}-a^{2}-b^{2},
\end{eqnarray*}%
from which we obtain $g_{33}=1+a^{2}+b^{2}$.

The components of the tensor field $\phi $ are immediate consequences of 
\begin{equation*}
\phi (\xi )=\phi (\frac{\partial }{\partial x})=0,\text{ \ \ }\phi (\frac{%
\partial }{\partial y})=\phi X=a\frac{\partial }{\partial x}+b\frac{\partial 
}{\partial y}+\frac{\partial }{\partial z},
\end{equation*}%
\begin{eqnarray*}
\phi (\frac{\partial }{\partial z}) &=&\phi (\phi X-a\frac{\partial }{%
\partial x}-b\frac{\partial }{\partial y})=\phi ^{2}X-a\phi (\frac{\partial 
}{\partial x})-b\phi (\frac{\partial }{\partial y}) \\
&=&-X-b(a\frac{\partial }{\partial x}+b\frac{\partial }{\partial y}+\frac{%
\partial }{\partial z}) \\
&=&-\frac{\partial }{\partial y}-ab\frac{\partial }{\partial x}-b^{2}\frac{%
\partial }{\partial y}-b\frac{\partial }{\partial z} \\
&=&-ab\frac{\partial }{\partial x}-(1+b^{2})\frac{\partial }{\partial y}-b%
\frac{\partial }{\partial z}.
\end{eqnarray*}%
The expression of the 1-form $\eta $, immediately follows from $\eta (\xi
)=1 $, $\eta (X)=\eta (\phi X)=0$%
\begin{equation*}
\eta =dx-adz.
\end{equation*}%
Now we calculate the components of the tensor field $h$ with respect to the
basis $\frac{\partial }{\partial x}$, $\frac{\partial }{\partial y}$, $\frac{%
\partial }{\partial z}$.%
\begin{equation*}
h(\xi )=h(\frac{\partial }{\partial x})=0,\text{ \ \ }h(\frac{\partial }{%
\partial y})=\lambda \frac{\partial }{\partial y},
\end{equation*}%
\begin{eqnarray*}
h(\frac{\partial }{\partial z}) &=&h(\phi X-a\frac{\partial }{\partial x}-b%
\frac{\partial }{\partial y}) \\
&=&h\phi X-ah(\frac{\partial }{\partial x})-bh(\frac{\partial }{\partial y})
\\
&=&-\lambda \phi X-b\lambda \frac{\partial }{\partial y} \\
&=&-\lambda (a\frac{\partial }{\partial x}+b\frac{\partial }{\partial y}+%
\frac{\partial }{\partial z})-b\lambda \frac{\partial }{\partial y},
\end{eqnarray*}%
\begin{equation*}
h(\frac{\partial }{\partial z})=-\lambda a\frac{\partial }{\partial x}%
-2b\lambda \frac{\partial }{\partial y}-\lambda \frac{\partial }{\partial z}.
\end{equation*}%
Thus the proof of the Case V is completed.

Case VI): $\mu =2(1-\sqrt{1-\kappa }).$

As in the Case V, we consider an $h$-frame $\{\xi ,X,\phi X\}$. Using the
assumption $\mu =2(1-\sqrt{1-\kappa })$ and (\ref{zetaAB2}) we obtain $B=0$
and thus the relation (\ref{eq5}) is written as%
\begin{equation}
\lbrack \xi ,X]=2\lambda \phi X,  \label{eq4.9}
\end{equation}%
\begin{equation}
\text{\ }[\xi ,\phi X]=0,\text{ \ }  \label{eq4.10}
\end{equation}%
\begin{equation}
\text{\ }[X,\phi X]=\frac{A}{2\lambda }\phi X+2\xi .  \label{eq4.11}
\end{equation}%
Because of (\ref{eq4.10}) we find that there is a chart $(V^{\prime
},(x,y,z))$ such that 
\begin{equation*}
\xi =\frac{\partial }{\partial x},\text{ \ \ \ }\phi X=\frac{\partial }{%
\partial y}
\end{equation*}%
on $V^{\prime }$. We put%
\begin{equation*}
X=a\frac{\partial }{\partial x}+b\frac{\partial }{\partial y}+c\frac{%
\partial }{\partial z},
\end{equation*}%
where $a,b,c$ are smooth functions defined on $V^{\prime }.$ As in the Case
V, we can directly calculate the tensor fields $\eta $, $\phi $, $g$ and $h$
with respect to the basis $\frac{\partial }{\partial x}$, $\frac{\partial }{%
\partial y}$, $\frac{\partial }{\partial z}$. \ This completes the proof of
the mainTheorem.
\end{proof}

In the following Theorem, we will locally \ construct $(\kappa ,\mu
,\upsilon =const.\neq 0)$-contact metric manifolds with $\kappa <1$ and $\xi
(I_{M})=0$.

\begin{theorem}
Let $\kappa :I=I_{1}\times I_{2}\subset 
\mathbb{R}
^{2}\rightarrow 
\mathbb{R}
$ be a smooth function defined on open subset $I$ of $%
\mathbb{R}
^{2}$, such that $\kappa (x,z)=1-(r(z)e^{vx})^{2}<1$ for any $(x,z)\in I,$
where $r(z)$ is a smooth function on open interval $I_{2}$ and $\upsilon $
is constant different from zero. Then we can construct two families of
non-Sasakian $(\kappa _{i},\mu _{i},v)$-manifolds $(M_{i},\phi _{i},\xi
_{i},\eta _{i},g_{i})$, $i=1,2$, in the set $M=I\times 
\mathbb{R}
\subset 
\mathbb{R}
^{3}$, so that for any $P(x,z,y)\in M$, the following are valid.%
\begin{equation*}
\kappa _{1}(P)=\kappa _{2}(P)=\kappa (x,z),\text{ \ \ }\mu _{1}(P)=2(1+\sqrt{%
1-\kappa (x,z)}\text{ \ \ and \ \ }\mu _{2}(P)=2(1-\sqrt{1-\kappa (x,z)}
\end{equation*}

Each family is determined by two arbitrary smooth functions of two variables.

\begin{proof}
We put $\lambda (x,z)=\sqrt{1-\kappa (x,z)}=r(z)e^{vx}>0$ and \ consider on $%
M$ the linearly independent vector fields%
\begin{equation}
\xi _{1}=\frac{\partial }{\partial x},\text{ \ \ }X_{1}=\frac{\partial }{%
\partial y}\text{ and }Y_{1}=a\frac{\partial }{\partial x}+b\frac{\partial }{%
\partial y}+\frac{\partial }{\partial z},  \label{eq4.12}
\end{equation}%
where $a(x,y,z)=2y+f(z)$, $b(x,y,z)=-\frac{y}{2}(y\upsilon +\upsilon f(z)+%
\frac{r^{\prime }(z)}{r(z)})+\frac{2}{\upsilon }r(z)e^{\upsilon x}+s(z)$, $%
f(z),$ $r(z),$ $s(z)$ are arbitrary smooth functions of $z$. The structure
tensor fields $\eta _{1},g_{1},\phi _{1}$ are defined by $\eta _{1}=$ $%
dx-(2y+f(z))dz$, $g_{1}=\left( 
\begin{array}{ccc}
1 & 0 & -a \\ 
0 & 1 & -b \\ 
-a & -b & 1+a^{2}+b^{2}%
\end{array}%
\right) $ and $\phi _{1}=\left( 
\begin{array}{ccc}
0 & a & -ab \\ 
0 & b & -1-b^{2} \\ 
0 & 1 & -b%
\end{array}%
\right) $, respectively. From (\ref{eq4.12}), we can easily obtain 
\begin{eqnarray}
\lbrack \xi _{1},X_{1}] &=&0\text{, \ \ }[\xi _{1},Y_{1}]=2\lambda (x,z)X_{1}%
\text{, \ \ }  \label{4.12a} \\
\lbrack X_{1},Y_{1}] &=&-\frac{[(2y+f(z))\upsilon r(z)+r^{\prime
}(z)]e^{\upsilon x}}{2\lambda (x,z)}X_{1}+2\xi _{1}.  \label{4.12b}
\end{eqnarray}%
Since $\eta _{1}\wedge d\eta _{1}=-2dx\wedge dy\wedge dz\neq 0$ everywhere
on $M$, we decide that $\eta _{1}$ is a contact form. By using just defined $%
g_{1}$ and $\phi _{1}$, we find\ $\eta _{1}=g(.,\xi _{1}),$ $\phi
_{1}X_{1}=Y_{1}$, $\phi _{1}Y_{1}=-X_{1}$, $\phi _{1}\xi _{1}=0$ and $d\eta
_{1}(Z,W)=g_{1}(Z,\phi _{1}W)$, $g_{1}(\phi _{1}Z,\phi
_{1}W)=g_{1}(Z,W)-\eta _{1}(Z)\eta _{1}(W)$ for any $Z$, $W\in \Gamma (M)$.
From the well known Koszul's formula and (\ref{B2}), we obtain%
\begin{equation}
\nabla _{X_{1}}\xi _{1}=-(\lambda (x,z)+1)Y_{1},\text{ \ \ }\nabla
_{Y_{1}}\xi =(1-\lambda (x,z))X_{1},  \label{eq4.13}
\end{equation}%
\begin{equation}
\nabla _{\xi _{1}}\xi _{1}=0,\text{ \ \ }\nabla _{\xi _{1}}X_{1}=-(1+\lambda
(x,z))Y_{1},\text{ \ \ }\nabla _{\xi _{1}}Y_{1}=(1+\lambda (x,z))X_{1},
\label{eq414}
\end{equation}%
\begin{equation}
\nabla _{X_{1}}X_{1}=\frac{[(2y+f(z))\upsilon r(z)+r^{\prime
}(z)]e^{\upsilon x}}{2\lambda (x,z)}Y_{1},\text{ \ \ }\nabla _{Y_{1}}Y_{1}=0,
\label{eq4.15}
\end{equation}%
\begin{eqnarray}
\nabla _{Y_{1}}X_{1} &=&(\lambda (x,z)-1)\xi _{1},\text{ \ \ }
\label{eq4.16} \\
\nabla _{X_{1}}Y_{1} &=&-\frac{[(2y+f(z))\upsilon r(z)+r^{\prime
}(z)]e^{\upsilon x}}{2\lambda (x,z)}X_{1}+(\lambda (x,z)+1)\xi _{1},
\label{eq4.17}
\end{eqnarray}%
$h_{1}\phi _{1}X_{1}=-\lambda (x,z)\phi _{1}X_{1}$ and $h_{1}X_{1}=\lambda
(x,z)X_{1}$, where $\nabla $ is Levi-Civita connection of $g_{1}$. By using
\ the relations (\ref{eq4.13})-(\ref{eq4.17}) we obtain 
\begin{eqnarray*}
R(X_{1},\xi _{1})\xi _{1} &=&\kappa _{1}X_{1}+\mu _{1}h_{1}X_{1}+v\phi
_{1}h_{1}X_{1},\text{ \ } \\
\text{\ }R(Y_{1},\xi _{1})\xi _{1} &=&\kappa _{1}Y_{1}+\mu
_{1}h_{1}Y_{1}+v\phi _{1}h_{1}Y_{1}, \\
R(X_{1,}Y_{1})\xi _{1} &=&0.
\end{eqnarray*}%
From the above relations and by virtue of the linearity of \ the curvature
tensor $R$, we conclude that 
\begin{equation*}
R(Z,W)\xi _{1}=(\kappa _{1}I+\mu _{1}h_{1}+\upsilon \phi _{1}h_{1})(\eta
_{1}(Z)W-\eta _{1}(W)Z)
\end{equation*}%
for any $Z,W\in \Gamma (M),$ i.e. $(M,\phi _{1},\xi _{1},\eta _{1},g_{1})$
is $(\kappa _{1},\mu _{1},\upsilon =const.)$ contact metric manifold with $%
\xi (I_{M})=0$ and thus the construction of the first \ family is completed.
For the second construction, we consider the vector fields 
\begin{equation}
\xi _{2}=\frac{\partial }{\partial x},\text{ \ \ }Y_{2}=\frac{\partial }{%
\partial y}\text{ },
\end{equation}%
\begin{equation}
X_{2}=(-2y+f(z))\frac{\partial }{\partial x}+(\frac{y^{2}}{2}\upsilon -y%
\frac{f(z)}{2}\upsilon -\frac{y}{2}\frac{r^{^{\prime }}(z)}{r(z)}+\frac{2}{%
\upsilon }r(z)e^{\upsilon x}+s(z))\frac{\partial }{\partial y}+\frac{%
\partial }{\partial z}
\end{equation}%
and define the tensor fields $\eta _{2},g_{2},\phi _{2},h_{2}$ as follows:%
\begin{equation*}
\eta _{2}=dx-(-2y+f(z))dz
\end{equation*}%
\begin{equation*}
\text{\ }g_{2}=\left( 
\begin{array}{ccc}
1 & 0 & -a \\ 
0 & 1 & -b \\ 
-a & -b & 1+a^{2}+b^{2}%
\end{array}%
\right) ,\text{ }\phi =\left( 
\begin{array}{ccc}
0 & -a & ab \\ 
0 & -b & 1+b^{2} \\ 
0 & -1 & b%
\end{array}%
\right) ,
\end{equation*}%
\begin{equation*}
\text{\ \ \ \ }h_{2}=\left( 
\begin{array}{ccc}
0 & 0 & a\lambda _{2} \\ 
0 & -\lambda _{2} & 2\lambda _{2}b \\ 
0 & 0 & \lambda _{2}%
\end{array}%
\right)
\end{equation*}%
with respect to the basis $\left( \frac{\partial }{\partial x},\frac{%
\partial }{\partial y},\frac{\partial }{\partial z}\right) ,$ where $%
a=-2y+f(z)$, $b=(\frac{y^{2}}{2}\upsilon -y\frac{f(z)}{2}\upsilon -\frac{y}{2%
}\frac{r^{^{\prime }}(z)}{r(z)}+\frac{2}{\upsilon }r(z)e^{\upsilon x}+s(z)$.
As in first construction, we say that $(M,\phi _{2},\xi _{2},\eta
_{2},g_{2}) $ is $(\kappa _{2},\mu _{2},\upsilon =const.)$-contact metric
manifold with $\xi (I_{M})=0$, where $\kappa _{2}(x,y,z)=\kappa
_{2}(x,z)=r(z)e^{vx}$, $\mu _{2}(x,y,z)=2(1-\sqrt{\kappa _{2}(x,z)})$. This
completes the proof of the Theorem.
\end{proof}
\end{theorem}

The Ricci operator $Q$ was given in the relation (\ref{Ricci Q}) for any $%
(\kappa ,\mu ,\upsilon )$-contact metric manifold $(M^{3},\phi ,\xi ,\eta
,g) $. If we carefully look at this relation, the scalar curvature $\tau $
is not obvious. Now, we will give the scalar curvature $\tau $ \ respect to $%
\kappa $, $\mu $ and $\upsilon $ for $(\kappa ,\mu ,\upsilon =const.)$%
-contact metric manifold.

\begin{theorem}
Let $(M,\phi ,\xi ,\eta ,g)$ be a non Sasakian $(\kappa ,\mu ,\upsilon
=const.)$-contact metric manifold. Then, 
\begin{equation*}
\bigtriangleup \lambda =X(A)+\phi X(B)+\upsilon ^{2}\lambda -\frac{1}{%
2\lambda }(A^{2}+B^{2})
\end{equation*}%
and 
\begin{equation*}
\tau =\frac{1}{\lambda }(\bigtriangleup \lambda -\upsilon ^{2}\lambda )-%
\frac{1}{\lambda ^{2}}\parallel grad\lambda \parallel ^{2}+2(\kappa -\mu ),
\end{equation*}%
where $\bigtriangleup \lambda $ is Laplacian of $\lambda $.
\end{theorem}

\begin{proof}
Using the definition of the Laplacian and together by Lemma 1, we have 
\begin{eqnarray*}
\bigtriangleup \lambda &=&XX(\lambda )+\phi X\phi X(\lambda )+\xi \xi
(\lambda ) \\
&&-(\bigtriangledown _{X}X)\lambda -(\bigtriangledown _{\phi X}\phi
X)\lambda -(\bigtriangledown _{\xi }\xi )\lambda \\
&=&X(A)+\phi X(B)+\upsilon ^{2}\lambda -\frac{1}{2\lambda }(A^{2}+B^{2}).
\end{eqnarray*}

For the computing scalar curvature $\tau $ of $M$, we will use (\ref{eq1})-(%
\ref{eq4}). Defining the curvature tensor $R$, we obtain%
\begin{eqnarray*}
R(X,\phi X)\phi X &=&\bigtriangledown _{X}\nabla _{\phi X}\text{ }\phi
X-\bigtriangledown _{\phi X}\nabla _{X}\text{ }\phi X-\nabla _{\left[ X,\phi
X\right] }\phi X \\
&=&\nabla _{X}\left( \frac{A}{2\lambda }X\right) -\nabla _{\phi X}\left( -%
\frac{B}{2\lambda }X+(1+\lambda )\xi \right) -\nabla _{-\frac{B}{2\lambda }X+%
\frac{A}{2\lambda }\phi X+2\xi }\phi X \\
&=&X\left( \frac{A}{2\lambda }\right) X+\frac{A}{2\lambda }\nabla _{X}X+\phi
X\left( \frac{B}{2\lambda }\right) X+\frac{B}{2\lambda }\nabla _{\phi X}X \\
&&-\phi X(\lambda )\xi -(1+\lambda )\nabla _{\phi X}\xi +\frac{B}{2\lambda }%
\nabla _{X}\phi X-\frac{A}{2\lambda }\nabla _{\phi X}\text{ }\phi X-2\nabla
_{\xi }\text{ }\phi X \\
&=&X\left( \frac{A}{2\lambda }\right) X+\frac{A}{2\lambda }\frac{B}{2\lambda 
}\phi X+\phi X\left( \frac{B}{2\lambda }\right) X \\
&&+\frac{B}{2\lambda }\left( -\frac{A}{2\lambda }\phi X+(\lambda -1)\xi
\right) \\
&&-\phi X(\lambda )\xi -(1+\lambda )(1-\lambda )X \\
&&+\frac{B}{2\lambda }\left( -\frac{B}{2\lambda }X+(1+\lambda )\xi \right) -%
\frac{A}{2\lambda }\left( \frac{A}{2\lambda }X\right) -2\left( \frac{\mu }{2}%
X\right)
\end{eqnarray*}%
\begin{eqnarray*}
&=&\left[ X\left( \frac{A}{2\lambda }\right) +\phi X\left( \frac{B}{2\lambda 
}\right) -\frac{B^{2}}{4\lambda ^{2}}-\frac{A^{2}}{4\lambda ^{2}}+(\lambda
^{2}-1)-\mu \right] X \\
&=&\left[ \frac{1}{2}\left( \frac{X(A)\lambda -A^{2}}{\lambda ^{2}}+\frac{%
\phi X(B)\lambda -B^{2}}{\lambda ^{2}}\right) -\frac{1}{4\lambda ^{2}}%
(A^{2}+B^{2})+(\lambda ^{2}-1)-\mu \right] X \\
&=&\left[ \frac{1}{2}\frac{X(A)+\phi X(B)}{\lambda }-\frac{1}{2\lambda ^{2}}%
(A^{2}+B^{2})-\frac{1}{4\lambda ^{2}}(A^{2}+B^{2})+(\lambda ^{2}-1)-\mu %
\right] X \\
&=&\left[ \frac{1}{2\lambda }\left( X(A)+\phi X(B)-\frac{1}{2\lambda }%
(A^{2}+B^{2})\right) -\frac{1}{2\lambda ^{2}}(A^{2}+B^{2})+(\lambda
^{2}-1)-\mu \right] X \\
&=&\left[ \frac{1}{2\lambda }(\bigtriangleup \lambda -\upsilon ^{2}\lambda )-%
\frac{1}{2\lambda ^{2}}\parallel grad\lambda \parallel ^{2}-(\kappa +\mu )%
\right] X
\end{eqnarray*}%
and thus%
\begin{equation*}
g(R(X,\phi X)\phi X,X)=\frac{1}{2\lambda }(\bigtriangleup \lambda -\upsilon
^{2}\lambda )-\frac{1}{2\lambda ^{2}}\parallel grad\lambda \parallel
^{2}-(\kappa +\mu )
\end{equation*}

By definition of scalar curvature, i.e. $\tau =TrQ=g(QX,X)+g(Q\phi X,\phi
X)+g(Q\xi ,\xi ),$ and using (\ref{Ricci} ), we have%
\begin{eqnarray*}
\tau &=&2g(R(X,\phi X)\phi X,X)+2g(Q\xi ,\xi ) \\
&=&\frac{1}{\lambda }(\bigtriangleup \lambda -\upsilon ^{2}\lambda )-\frac{1%
}{\lambda ^{2}}\parallel grad\lambda \parallel ^{2}-2(\kappa +\mu )+4\kappa
\\
&=&\frac{1}{\lambda }(\bigtriangleup \lambda -\upsilon ^{2}\lambda )-\frac{1%
}{\lambda ^{2}}\parallel grad\lambda \parallel ^{2}+2(\kappa -\mu ).
\end{eqnarray*}%
Thus the proof of \ the Theorem is completed.
\end{proof}

\begin{remark}
Let us suppose that $\mu =2$. By (\ref{Xmu}) and (\ref{FiXmu}), we have \ $%
X(\lambda )=\phi X(\lambda )=0.$ Using \ this relation \ in \ (\ref{eq6}),
we obtain \ $[X,\phi X]=2\xi .$ But this relation says that $[X,\phi
X](\lambda )=0=2\xi (\lambda )=2\upsilon \lambda .$ This is a contradiction
with $\lambda \neq 0$ and $\upsilon =const.\neq 0$. Because of this fact, we
did not consider this case.
\end{remark}

\end{document}